\documentclass{amsart}
\usepackage{color}
\usepackage{mathabx}
\usepackage{mathrsfs}
\usepackage{enumerate}
 \newtheorem{thm}{Theorem}[section]
 \newtheorem{cor}[thm]{Corollary}

 \theoremstyle{definition}
 
 \theoremstyle{remark}

 \numberwithin{equation}{section}

\newcommand{\C}{\mathbb{C}}

\newcommand{\N}{\mathbb{N}}

\newcommand*{\Ann}{\ensuremath{\mathrm{Ann\,}}}
\newcommand*{\supp}{\ensuremath{\mathrm{supp\,}}}

\begin{document}

%
%
%
%
%
%
%
%
%

\title[]{Spectral synthesis via moment functions on  hypergroups}

\author{\.Zywilla Fechner }
\address{Institute of Mathematics,\\
				 Lodz University of Technology,\\
				90-924 {\L}\'{o}d\'{z}, \\
				ul. W\'{o}lcza\'{n}ska 215, Poland 
        \\ ORCID 0000-0001-7412-6544}
\email{zfechner@gmail.com }

\vbox{\author{Eszter Gselmann}
\address{University of Debrecen,\\
 H-4002 Debrecen,\\ P.O.Box: 400, Hungary}
\email{gselmann@science.unideb.hu}
}

\author{L\'{a}szl\'{o} Sz\'{e}kelyhidi}
\address{University of Debrecen,\\
 H-4002 Debrecen,\\ P.O.Box: 400, Hungary
\\ ORCID: 0000-0001-8078-6426}
\email{ szekely@science.unideb.hu, lszekelyhidi@gmail.com}

\thanks{
Project no.~K134191 supporting E.~Gselmann and L.~Sz\'{e}kelyhidi, has been implemented by the support provided from the National Research, Development and Innovation Fund of Hungary, financed under the K\_20 funding scheme.
The research of E.~Gselmann has partially been carried out with the help of the project 2019-2.1.11-T\'{E}T-2019-00049,
which has been implemented by the support provided from the National Research, Development
and Innovation Fund of Hungary. 
}



\date{\today}

\begin{abstract}
In this paper we continue the discussion about relations between exponential polynomials and generalized moment generating functions on a commutative hypergroup.  We are interested in the following problem: is it true that every finite dimensional variety is spanned by moment functions? Let $m$ be an exponential on $X$.
In our former paper we have proved that if the linear space of all $m$-sine functions in the variety of an $m$-exponential monomial is (at most) one dimensional, then this variety is spanned by moment functions generated by $m$. In this paper we show that this may happen also in cases where the $m$-sine functions span a more than one dimensional subspace in the variety. We recall the notion of a polynomial hypergroup in $d$ variables, describe exponentials on it and give the characterization of the so called $m$-sine functions. Next we show that the Fourier algebra of a polynomial hypergroup in $d$ variables is the polynomial ring in $d$ variables. Finally, using Ehrenpreis--Palamodov Theorem we show that every exponential polynomial on the polynomial hypergroup in $d$ variables is a linear combination of moment functions contained in its variety.
\end{abstract}

\maketitle

\section{Introduction}

In our former paper (see \cite{FecGseSze20AEQ}) we studied the following problem: given a commutative hypergroup $X$, is it true that all exponential polynomials are included in the linear span of all generalized moment functions? A more precise formulation of this problem reads as follows: given an exponential monomial corresponding to the exponential $m$ on the commutative hypergroup $X$, is it true that the (finite dimensional) variety of this exponential monomial has a basis consisting of generalized moment functions associated with the \hbox{exponential $m$}? In \cite{FecGseSze20AEQ}, we proved that a sufficient condition for this is that all $m$-sine functions in the given variety form a one dimensional subspace. Of course, the same holds if all $m$-sine functions in the variety are zero, as in that case every exponential monomial, as well as every generalized moment function in the given variety is a constant multiple of $m$. In addition, it was shown in \cite{FecGseSze20AEQ} that this sufficient 
condition holds on polynomial hypergroups in one variable, for example.
\vskip.2cm

A natural question is if the condition that the $m$-sine functions in the given variety form a linear space of dimension at most one is necessary. The subject of this paper is to study this problem, to give a negative answer, and to support the conjecture that, in fact, every exponential monomial is a linear combination of generalized moment functions in its variety. In particular, we show that if $X$ is a polynomial hypergroup in $d$ variables, then every exponential monomial is the linear combination of generalized moment functions, however, there are exponential monomials, whose variety contains more than one linearly independent sine functions, if $d>1$.
\vskip.2cm

A {\it hypergroup} is a locally compact Hausdorff space $X$ equipped with an involution and a convolution operation defined on the space of all bounded complex regular measures on $X$. For the formal definition, historical background and basic facts about hypergroups we refer to \cite{MR1312826}. In this paper $X$ denotes a commutative hypergroup with identity element $o$, involution $\widecheck{}$\,, and convolution $*$. For each function $f$ on $X$ we define the function $\widecheck{f}$ by $\widecheck{f}(x)=f(\widecheck{x})$.
\vskip.2cm

Given $x$ in $X$ we denote the point mass with support the singleton $\{x\}$ by $\delta_x$. The convolution $\delta_x*\delta_y$ is a compactly supported probability measure on $X$, and for each continuous function $h\colon X\to \mathbb{C}$ the integral
$$ 
\int_X h(t)d(\delta_x*\delta_y)(t)
$$ 
will be denoted by $h(x*y)$. Given $y$ in $X$ the function $x\mapsto h(x*y)$ is the {\it translate} of $h$ by $y$. 
\vskip.2cm

A set of continuous complex valued functions on $X$ is called {\it translation invariant}, if it contains all translates of its elements. A linear translation invariant subspace of all continuous complex valued functions is called a {\it variety}, if it is closed with respect to uniform convergence on compact sets. The smallest variety containing the given function $h$ is called the {\it variety of $h$}, and is denoted by $\tau(h)$. Clearly, it is the intersection of all varieties including $h$. A continuous complex valued function is called an {\it exponential polynomial}, if its variety is finite dimensional. An exponential polynomial is called an {\it exponential monomial}, if its variety is {\it indecomposable}, that is, it is not the sum of two proper subvarieties. The simplest nonzero exponential polynomial is the one having one dimensional variety: it consists of all constant multiples of a nonzero continuous function. Clearly, it is an exponential monomial. If we normalize that function by taking $1$ 
at $o$, then we have the concept of an exponential. Recall that $m$ is an {\it exponential} on $X$ if it is a non-identically zero continuous complex-valued function  satisfying $m(x*y)=m(x) m(y)$ for each $x,y$ in $X$. Exponential monomials and polynomials on commutative hypergroups have been introduced and characterized in \cite{MR3116655, MR3192009, FeSz21}. In the study of exponential polynomials the basic tool is the modified difference (see in \cite{MR3116655}). Here we briefly recall the notion of modified difference. For a given function $f$ in $\mathcal{C}(X)$, $y$ in $X$ and an exponential $m\colon X\to \mathbb{C}$ we define the modified difference:

$$\Delta_{m;y} * f(x)=f(x*y)-m(y)f(x) $$

for all $x$ in $X$. For $y_1,\dots,y_{N+1}$ in $X$ we define

$$\Delta_{m;y_1\dots,y_{N+1}} * f(x)= \Delta_{m;y_1\dots,y_{N}} *[\Delta m;y_{N+1} * f(x)] $$

for all $x$ in $X$. The following characterization theorem of exponential monomials is proved in \cite[Corollary 2.7]{FeSz21}.

\begin{thm}\label{findif}
Let $X$ be a commutative hypergroup. The continuous function $f:X\to\C$ is an exponential monomial if and only if there exists an exponential $m$ and a natural number $n$ such that 
$$
\Delta_{m;y_1,y_2,\dots,y_{n+1}}*f=0
$$ 
holds for each $y_1,y_2,\dots,y_{n+1}$ in $X$.
\end{thm}

If $f$ is nonzero, then $m$ is uniquely determined, and $f$ is called an {\it $m$-exponential monomial}, and its {\it degree} is the smallest $n$ satisfying the property in Theorem \ref{findif}. The degree of each exponential function is zero: in fact, every nonzero exponential monomial of degree zero is a constant multiple of an exponential. For exponential monomials of first degree sine functions provide an example: given an exponential $m$ on $X$ the continuous function $s:X\to\C$ is called an {\it $m$-sine function}, if
$$
s(x*y)=s(x)m(y)+s(y)m(x)
$$
holds for each $x,y$ in $X$. If $s$ is nonzero, then $m$ is uniquely determined, and its degree is $1$.
\vskip.2cm

Important examples for exponential polynomials are provided by the moment functions. 
Let $X$ be a commutative hypergroup, $r$ a positive integer, and for each multi-index $\alpha$ in $\N^r$, let $f_{\alpha}:X\to \C$ be continuous function, such that $f_{\alpha}\ne 0$ for $|\alpha|=0$. We say that $(f_{\alpha})_{\alpha \in \mathbb{N}^{r}}$ is a \emph{generalized 
moment sequence of rank $r$}, if 
\begin{equation}\label{Eq3}
f_{\alpha}(x*y)=\sum_{\beta\leq \alpha} \binom{\alpha}{\beta} f_{\beta}(x)f_{\alpha-\beta}(y)
\end{equation}
holds whenever $x,y$ are in $X$ (see \cite{FecGseSze20}). It is obvious that the variety of $f_{\alpha}$ is finite dimensional: in fact, every translate of $f_{\alpha}$ is a linear combination of the finitely many functions $f_{\beta}$ with $\beta\leq \alpha$, by equation \eqref{Eq3}. The functions in a generalized moment function sequence are called {\it generalized moment functions}. In this paper we shall omit the "generalized" adjective: we simply use the terms {\it moment function sequence} and {\it moment function}. The order of $f_{\alpha}$ in the above moment function sequence is defined as $|\alpha|$.  It follows that moment functions are exponential polynomials. 
\vskip.2cm

The special case of rank $1$ moment sequences leads to the following system of functional equations: 
\begin{equation*}
f_k(x*y)=\sum_{j=0}^k {k \choose j}f_j(x)f_{k-j}(y)
\label{eq:Moment}
\end{equation*}
for all $k=0,1,\dots, N$ and for each $x,y$ in $X$. 
\vskip.2cm

Observe that in every moment function sequence the unique function $f_{(0,0,\dots,0)}$ is an exponential on the hypergroup $X$. In this case we say that the exponential
$m=f_{(0,0,\dots,0)}$ \textit{generates the given moment function sequence}, and that the moment functions in this sequence {\it correspond to $m$}.We may also say that the moment functions in the given moment sequence are {\it $m$-moment functions}. It is also easy to see, that in every moment function sequence $f_{\alpha}$ with $|\alpha|=1$ is an $m$-sine function on $X$, where $m=f_{(0,0,\dots,0)}$. 
\vskip.2cm

The main result in  \cite{FecGseSze20} is that moment sequences of rank $r$ can be described by using Bell polynomials if the underlying hypergroup is an Abelian group. The point is that in this case the situation concerning \eqref{Eq3} can be reduced to the case where the generating exponential is the identically $1$ function and the problem reduces to a problem about polynomials of additive functions. Unfortunately, if $X$ is not a group, then such a reduction is not available, in general.
\vskip.2cm

The following result is important.

\begin{thm}\label{mommon}
Let $X$ be a commutative hypergroup. Then the variety of a nonzero moment function contains exactly one exponential: the generating exponential function.
\end{thm}

\begin{proof}
In the moment function sequence in \eqref{Eq3} we show that the only exponential in the variety of $f=f_{\alpha}$ is $m=f_{(0,0,\dots,0)}$. First we show that $f$ is an $m$-exponential monomial of degree at most $|\alpha|$. We prove by induction on $N=|\alpha|$, and the statement is obviously true for $N=0$. Assuming that the statement holds for $|\alpha|\leq N$, we prove it for $|\alpha|= N+1$. By equation \eqref{Eq3}, we have
\begin{equation}\label{difmom}
\Delta_{m;y}*f_{\alpha}(x)=f_{\alpha}(x*y)-m(y)f_{\alpha}(x)=
\end{equation}
$$
f_{\alpha}(x*y)-f_{(0,0,\dots,0)}(y)f_{\alpha}(x)=\sum_{\beta<\alpha} \binom{\alpha}{\beta} f_{\beta}(x) f_{\alpha-\beta}(y).
$$
for each $x,y$ in $X$. The right side, as a function of $x$, is an exponential monomial of degree at most $N=|\alpha|-1$, corresponding to the exponential $m$, by the induction hypothesis. It follows that 
$$
 \Delta_{m;y_1,y_2,\dots,y_{N+1},y}*f_{\alpha}(x)=\Delta_{m;y_1,y_2,\dots,y_{N+1}}*[\Delta_{m;y}*f_{\alpha}](x)=
$$
$$
\sum_{\beta<\alpha} \binom{\alpha}{\beta} [\Delta_{m;y_1,\dots,y_{N+1}}*f_{\beta}(x)] f_{\alpha-\beta}(y)=0,
$$
which proves our first statement.
\vskip.2cm

To prove that $m$ is the only exponential in the variety of $f$ we know, from the results in \cite{MR4217968}, that the annihilator $\Ann \tau(m)$ of the variety $\tau(m)$ of $m$ is a closed maximal ideal $M_m$ in $\mathcal M_c(X)$,  the space of all compactly supported complex Borel measures on $X$, in which the measures $\Delta_{m;y}$ span a dense subspace, and $\tau(m)\subseteq \tau(f)$, hence $\Ann \tau(f)\subseteq M_m$. On the other hand, by the above equation, 
$$
M_m^{N+2}\subseteq \Ann \tau(f),
$$
as the product of any $N+2$ generators of a dense subspace of $M_m$ annihilates the function $f$. It follows that, if an exponential $m_1$ is in $\tau(f)$, then clearly  $\tau(m_1)\subseteq \tau(f)$, hence $\Ann \tau(f)\subseteq \Ann \tau(m_1)$. We conclude that 
$$
M_m^{N+2}\subseteq \Ann \tau(f)\subseteq\Ann \tau(m_1).
$$
As $\Ann \tau(m_1)$ is a maximal ideal, hence it is a prime ideal, as well, consequently we have $M_m\subseteq \Ann \tau(m_1)$. By maximality, it follows that we have $\tau(m)=\tau(m_1)$, which  immediately implies $m_1=m$.
\end{proof}

This theorem can be formulated in the way that every $m$-moment function is an $m$-exponential monomial.
\vskip.2cm

Exponential monomials are the basic building blocks of spectral synthesis. We say that \textit{spectral analysis} holds for a variety, if every nonzero subvariety in the variety contains a nonzero exponential monomial. We say that a variety is {\it synthesizable} if all exponential monomials in the variety span a dense subspace. We say that {\it spectral synthesis holds for} a variety if every subvariety of it is synthesisable. It is obvious that spectral synthesis for a variety implies spectral analysis for it.  If spectral analysis holds for every variaty on $X$, then we say that \textit{spectral analysis holds on $X$}. If every variety on $X$ is  synthesisable, then we say that {\it spectral synthesis holds on $X$}. Clearly, on every commutative hypergroup, spectral synthesis holds for finite dimensional varieties. 
\vskip.2cm

In the light of these definitions our basic problem about the relation between exponential monomials and moment functions can be reformulated: is it true that every finite dimensional variety is spanned by moment functions? If so, then spectral synthesis is rather based on moment functions, not on exponential monomials. Obviously, it is enough to consider indecomposable varieties, and even we may assume that the variety in question is the variety of an exponential monomial. Our result in \cite[Theorem 2.1]{FecGseSze20AEQ} says that if the linear space of all $m$-sine functions in the variety of an $m$-exponential monomial is (at most) one dimensional, then this variety is spanned by moment functions -- obviously, associated with $m$. In the subsequent paragraphs we show that this may happen also in cases where the $m$-sine functions span a more than one dimensional subspace in the variety. 

\section{Polynomial hypergroups}

In this section we recall the definition of a class of hypergroups which will be used in the sequel.
\vskip.2cm

The following definition is taken from \cite[Chapter 3, Section 3.1, p. 133]{MR1312826}. Let $d$ denote a positive integer, let $\mathcal P_n$ denote the set of polynomials in on $\C^d$ of degree at most $n$. Finally, let $\pi$ denote a probability measure on $\C^d$, and we assume that the commutative hypergroup $X$ with convolution $*$, \hbox{involution $\widecheck{}\,$}, and identity $o$ satisfies the following properties: for each $x$ in $X$ there exists a polynomial $Q_x$ on $\C^d$ such that we have:
\begin{enumerate}[(P1)]
	\item  For each nonnegative integer $n$, the set $X_n=\{x:\, x\in X, Q_x\in \mathcal P_n\}$ is a basis of the linear space $\mathcal P_n$.
    \item We have $Q_x(1,1,\dots,1)=1$ for each $x$ in $X$. 
    \item For each $x$ in $X$, we have $Q_{\widecheck{x}}(z)=\overline{Q}_x(z)$ whenever $z$ is in $\supp \pi$.
    \item For each $x,y,w$ in $X$, we have
    $$
    \int_{\C^d} Q_x Q_y Q_w\,d\pi\geq 0.
    $$
\end{enumerate}
In fact, the properties (P1), (P2), (P3), (P4) are not independent -- for the details see \cite[Proposition 3.1.4]{MR1312826}. Property (P4) expresses that the polynomials $Q_x$ for $x$ in $X$ are orthogonal with respect to the measure $\pi$.
\vskip.2cm

In this case, by property (P1), for each $x,y$ in $X$ the polynomial $Q_xQ_y$ admits a unique representation
\begin{equation}\label{linform}
Q_x Q_y=\sum_{w\in X} c(x,y,w) Q_w
\end{equation}
with some complex numbers $c(x,y,w)$. The formula \eqref{linform} is called \textit{linearization formula}, and the numbers $c(x,y,w)$ are called \textit{linearization coefficients}. 
\vskip.2cm

The hypergroup $X$ with convolution $*$ is called a \textit{polynomial hypergroup (in $d$ variables)} if there exists a family $\{Q_x:\,x\in X\}$ of polynomials satisfying the above property and
$$
\delta_x*\delta_y=\sum_{w\in X} c(x,y,w)\delta_w
$$ 
for each $x,y$ in $X$. In this case we say that the polynomial hypergroup $X$ is {\it associated with the family} of polynomials $\{Q_x:\,x\in X\}$. 
\vskip.2cm

The choice $d=1$ is the case of \textit{polynomial hypergroups in one variable}. In this case we may suppose that $X=\N$, the set of natural numbers, and the family of polynomials $\{Q_x:\,x\in X\}$ is in fact a sequence of polynomials $(P_n)_{n\in\N}$, where $P_n$ is of degree $n$ and we always assume that $P_0=1$. Suppose that there exists a Borel measure on a finite interval such that the sequence $(P_n)_{n\in\N}$ is orthogonal with respect to this measure. In this case it turns out that the linearization coefficients $c(k,l,n)$ are zero unless $|k-l|\leq n\leq k+l$. Hence the linearization formula has the form
\begin{equation}\label{linform1}
	P_k P_l=\sum_{l=|k-l|}^{k+l} c(k,l,n) P_n.
\end{equation}
In this situation, the sequence $(P_n)_{n\in\N}$ generates a polynomial hypergroup if and only if the linearization coefficients are nonnegative. 
\vskip.2cm

As a particular example we consider the sequence $(T_n)_{n\in\N}$ of Chebyshev polynomials of the first kind defined by the recursion $T_0(x)=1$, $T_1(x)=x$ and
$$
xT_n(x)=\frac{1}{2}[T_{n-1}(x)+T_{n+1}(x)]
$$
for $n=1,2,\dots$  In this case it is easy to check that the linearization coefficients are nonnegative and the resulting hypergroup is the \textit{Chebyshev hypergroup}. For more about polynomial hypergroups in one variable see also \cite{MR1312826,MR2978690}.
\vskip.2cm

Based on \cite[3.1.3. (iii)]{MR1312826}, this construction can be generalized for dimension $d>1$.
Indeed, if $K^{(1)}$ and $K^{(2)}$ are polynomial hypergroups, then $K^{(1)}\times K^{(2)}$ can be made into a polynomial hypergroup, as well. The collection 
\[
 \left\{Q_{x}\otimes Q_{y}\, : \, (x, y)\in K^{(1)}\times K^{(2)}\right\}
\]
of polynomials on $\mathbb{C}^{d_{1}+d_{2}}$ defines a hypergroup structure on 
$K^{(1)}\times K^{(2)}$, where $\left\{Q_{x}^{(1)}\, : \, x\in K^{(1)} \right\}$ and 
$\left\{Q_{y}^{(2)}\, : \, y\in K^{(2)} \right\}$ denote the collections of polynomials on 
$\mathbb{C}^{d_{1}}$ and $\mathbb{C}^{d_{2}}$ defining $K^{(1)}$ and $K^{(2)}$, respectively. 

Thus, the notion of Chebyshev hypergroups can also be extended for $d\geq 2$. 

For example, in case $d=2$, after the above construction, a hypergroup arises that can also be defined with the following recursion. 
Let $T_{0,0}(x,y)=1$, $T_{1,0}(x,y)=x$, $T_{0,1}(x,y)=y$, $T_{1,1}(x,y)=xy$, 
for each nonnegative integers $k,n$, let
\[
T_{k, n}(x, y)= 
\begin{cases}
 T_{k}(x), & \text{if } n=0\\
 T_{n}(y), & \text{if } k=0, \\
\end{cases}
\]
furthermore, for each positive integers $k, n$
with $k+n\geq 2$, 
\begin{equation}\label{Ch2}
xyT_{k,n}(x,y)=
\end{equation}
$$
\frac{1}{4}[T_{k-1,n-1}(x,y)+T_{k+1,n-1}(x,y)+T_{k-1,n+1}(x,y)+T_{k+1,n+1}(x,y)].
$$
Clearly, this is satisfied by $T_{k,n}(x,y)=T_k(x)\cdot T_n(y)$, which leads to a two dimensional generalization of the Chebyshev hypergroup on the basic set $X=\N\times \N$. 
\vskip.2cm

From the above recursion formulas we can derive
$$
T_{k,l}\cdot T_{m,n}=\frac{1}{4}[T_{|k-m|,|l-n|}+T_{|k-m|,l+n}+T_{k+m,|l-n|}+T_{k+m,l+n}],
$$
which provides us the convolution of point masses in $X$:
\begin{equation}\label{ConvCh2}
\delta_{k,l}* \delta_{m,n}=\frac{1}{4}[\delta_{|k-m|,|l-n|}+\delta_{|k-m|,l+n}+\delta_{k+m,|l-n|}+\delta_{k+m,l+n}]
\end{equation} 
whenever $k,l,m,n$ are natural numbers. The involution on $X$ is the identity mapping, and the identity is $\delta_{0,0}$.
\vskip.2cm

Exponential functions on $X$ are described in the following theorem (see \cite[Proposition 3.1.2]{MR1312826}, \cite[Theorem 3.1]{MR2978690}).

\begin{thm}\label{expX}
The function $M:\N\times \N\to\C$ is an exponential on $X$ if and only if there exist complex numbers $\lambda,\mu$ such that
\begin{equation}\label{Mform}
M(x,y)=T_x(\lambda)T_y(\mu)
\end{equation}
holds whenever $(x,y)$ are in $X$. The correspondence between the exponentials $M$ and the pairs $(\lambda,\mu)$ is one-to-one.
\end{thm}

\begin{proof}
If $M$ has the given form with some complex numbers $\lambda, \mu$, then for each natural numbers $k,l,m,n$ we have
{\small
$$
M[\delta_{k,l}* \delta_{m,n}]=\int_X M(u,v)\,d(\delta_{k,l}* \delta_{m,n})(u,v)=
$$
$$
\frac{1}{4}[T_{|k-m|}(\lambda)T_{|l-n|}(\mu)+T_{|k-m|}(\lambda)T_{l+n}(\mu)+T_{k+m}(\lambda)T_{|l-n|}(\mu)+T_{k+m}(\lambda)T_{l+n}(\mu)]=
$$
$$
T_{k,l}(\lambda)T_{m,n}(\mu)=T_k(\lambda)T_l(\mu)\cdot T_m(\lambda)T_n(\mu)=M(\delta_{k,l})M(\delta_{m,n}),
$$
}

\noindent further $M(\delta_{0,0})=T_0(\lambda)T_0(\mu)=1$, hence $M$ is an exponential on the hypergroup $X$.
\vskip.2cm

Conversely, suppose that $M:\N\times\N\to\C$ is an exponential on $X$, that is, we have
\begin{equation}\label{Mexp1}
M[\delta_{k,l}* \delta_{m,n}]=M(\delta_{k,l})M(\delta_{m,n})
\end{equation}
for each natural numbers $k,l,m,n$, further $M(\delta_{0,0})=1$. We define the function
$$
f(m,n)=M(\delta_{m,n})
$$
whenever $m,n$ are in $\N$. Then $f$ satisfies
{\small
\begin{equation}\label{Mexp2}
f(k,l)f(m,n)=
\end{equation}
$$
\frac{1}{4}[M(\delta_{|k-m|,|l-n|})+M(\delta_{|k-m|,l+n})+M(\delta_{k+m,|l-n|})+M(\delta_{k+m,l+n})]=
$$
}
{\small
$$
\frac{1}{4}[f(|k-m|,|l-n|)+f(|k-m|,l+n)+f(k+m,|l-n|)+f(k+m,l+n)]
$$
}
for each $k,l,m,n$ in $\N$. Here we substitute $l=m=0$ to get
\begin{equation}\label{Mexp3}
f(k,0)f(0,n)=f(k,n),
\end{equation}
and the substitution $l=n=0$ in \eqref{Mexp2} gives
$$
f(k,0)f(m,0)=\frac{1}{2}[f(|k-m|,0)+f(k+m,0)]
$$
for all $k,m$ in $\N$. As $f(0,0)=1$, the function $k\mapsto f(k,0)$ is an exponential of the Chebyshev hypergroup (see \cite{MR2978690}), hence $f(k,0)=T_k(\lambda)$ for each $k$ in $\N$ with some complex number $\lambda$. Similarly, we have that $f(0,n)=T_n(\mu)$ holds for each $n$ in $\N$ with some complex number $\mu$. It follows, by \eqref{Mexp3}, that
$$
M(\delta_{k,n})=f(k,n)=f(k,0)f(0,n)=T_k(\lambda)T_n(\mu),
$$
and our statement is proved. Clearly, the pair $\lambda,\mu$ is uniquely determined by the exponential 
$M$.
\end{proof}

Using this theorem it is reasonable to denote the exponential corresponding to the pair $(\lambda,\mu)$ in $\C^2$ by $M_{\lambda,\mu}$. 

\begin{thm}\label{sineX}
Let $M_{\lambda,\mu}:\N\times\N\to\C$ be an exponential on $X$.
The function $S:\N\times\N\to\C$ is an $M_{\lambda,\mu}$-sine function on $X$ if and only if there are complex numbers $a,b$ such that 
\begin{equation}\label{Msine}
S(x,y)=aT_x'(\lambda)T_y(\mu)+bT_x(\lambda)T_y'(\mu)
\end{equation}
whenever $x,y$ are in $X$.
\end{thm}

\begin{proof}
The function $S$ given in \eqref{sineX} is an $M_{\lambda,\mu}$-sine function on $X$, as it is easy to check by simple calculation. For the converse we assume that \hbox{$S:\N\times\N\to\C$} is an $M_{\lambda,\mu}$-sine function on $X$. Then we have for each $x,y,u,v$ in $\N$:
\begin{equation}\label{S0}
S\bigl((x,y)*(u,v)\bigr)=S(x,y)M_{\lambda,\mu}(u,v)+S(u,v)M_{\lambda,\mu}(x,y).
\end{equation}
Using the form of the exponentials on $X$ given in the previous Theorem \ref{expX} and substituting $y=u=0$ we get
\begin{equation}\label{S1}
S(x,v)=S(x,0)T_v(\mu)+S(0,v)T_x(\lambda).
\end{equation}
Now we substitute $y=v=0$ in \eqref{S0} to get
\begin{equation}\label{S2}
\frac{1}{2}[S(|x-u|,0)+S(x+u,0)]=S(x,0)T_u(\lambda)+S(u,0)T_x(\lambda).
\end{equation}
Using the fact that the function $M_{\lambda}:x\mapsto T_x(\lambda)$ is an exponential on the Chebyshev hypergroup (see \cite{MR2978690}), we can write equation \eqref{S2} in the form
\begin{equation}\label{S3}
	S(x*u,0)=S(x,0)M_{\lambda}(u)+S(u,0)M_{\lambda}(x),
\end{equation}
and this means that the function $x\mapsto S(x,0)$ is an $M_{\lambda}$-sine function on the Chebyshev hypergroup. It follows from \cite[Theorem 2.5.]{MR2978690}, that
$$
S(x,0)=a \frac{d}{d\lambda}M_{\lambda}(x)=a T'_x(\lambda)
$$
with some complex number $a$. Similarly, we have that
$$
S(0,v)=b \frac{d}{d\mu}M_{\mu}(x)=a T'_v(\mu).
$$
Finally, substitution into \eqref{S1} gives the statement.
\end{proof}

It follows that the linear space of all $M_{\lambda,\mu}$-sine functions on $X$ is at most two dimensional, it is spanned by the two functions $(x,y)\mapsto T_x'(\lambda)T_y(\mu)$ and $(x,y)\mapsto T_x(\lambda)T_y'(\mu)$. On the other hand, these two functions are linearly independent. Indeed, assume that
$$
\alpha T_x'(\lambda)T_y(\mu)+\beta T_x(\lambda)T_y'(\mu)=0
$$
holds for some complex numbers $\alpha, \beta$ and for all $x,y$ in $X$. Substituting $x=0, y=1$ and using $T_0(z)=1$ and $T_1(z)=z$ we have $T_0'(\lambda)=0$ and $T_1'(\mu)=1$, hence we get $\beta=0$. Interchanging the role of $x$ and $y$, we obtain $\alpha=0$, hence the two functions $(x,y)\mapsto T_x'(\lambda)T_y(\mu)$ and $(x,y)\mapsto T_x(\lambda)T_y'(\mu)$ are linearly independent -- they form a basis of the linear space of all $M_{\lambda,\mu}$-sine functions. We consider the variety $\tau(\varphi)$ of an $M_{\lambda,\mu}$-exponential monomial, which is not a linear combination of moment functions in $\tau(\varphi)$ associated with the exponential $M_{\lambda,\mu}$. From \cite[Theorem 2.1]{FecGseSze20AEQ} it follows, that the two functions $(x,y)\mapsto T_x'(\lambda)T_y(\mu)$ and $(x,y)\mapsto T_x(\lambda)T_y'(\mu)$ must belong to $\tau(\varphi)$, hence the linear space of all $M_{\lambda,\mu}$-sine functions in $\tau(\varphi)$ is two dimensional. We will show that still $\tau(\varphi)$ is generated 
by moment functions. This would verify that the sufficient condition given in \cite[Theorem 2.1]{FecGseSze20AEQ} for that the variety of an exponential monomial is generated by moment functions is not necessary. In fact, we will show that, in general, on every polynomial hypergroup, every finite dimensional variety is spanned by moment functions. In particular, the variety of each exponential monomial is spanned by moment functions -- it follows that every exponential monomial is a linear combination of  moment functions. On the other hand, it is easy to see that on polynomial hyeprgroups in more than one variable there are exponential monomials whose variety contains more than one linearly independent sine function.
\vskip.2cm

We note that the formula for $S$ in equation \eqref{Msine} can be written in the following form
$$
S(x,y)=a \partial_{\lambda} M_{\lambda,\mu}(x,y)+b \partial_{\mu} M_{\lambda,\mu}(x,y),
$$
which is a special case of the moment functions appear in Theorem \ref{poldifmom} in the next section. 

\section{A class of moment functions on polynomial hypergroups}

Let $X$ be a polynomial hypergroup in $d$ variables associated with the family of polynomials $\{Q_x:\, x\in X\}$, and we always assume that $Q_o=1$, where $o$ is the identity of $X$. We introduce a special class of exponential monomials \hbox{on $X$.} Let $P$ be a polynomial in $d$ variables: 
$$
P(\xi)=\sum_{|\alpha|\leq N} a_{\alpha} \xi^{\alpha},
$$
where we use multi-index notation. This means that $\xi=(\xi_1,\xi_2,\dots \xi_d)$ is in $\C^d$, $\alpha=(\alpha_1,\alpha_2,\dots,\alpha_d)$ is in $\N^d$, $a_{\alpha}$ is a complex number, and
$$
|\alpha|=\alpha_1+\alpha_2+\dots+\alpha_d,\enskip \xi^{\alpha}=\xi_1^{\alpha_1} \xi_2^{\alpha_2}\cdots \xi_d^{\alpha_d}.
$$
Then we write
$$
P(\partial)=\sum_{|\alpha|\leq N} a_{\alpha} \partial^{\alpha},
$$
where $\partial=(\partial_1,\partial_2,\dots,\partial_d)$ with the obvious meaning of the partial differential operators $\partial_i$. The differential operator $P(\partial)$ is defined on the space of polynomials in $d$ variables in the usual way: given the polynomial $Q$ in the polynomial ring $\C[z_1,z_2,\dots,z_d]$ then
$$
[P(\partial)Q](\xi)=\sum_{|\alpha|\leq N} a_{\alpha} [\partial_1^{\alpha_1} \partial_2^{\alpha_2}\cdots \partial_d^{\alpha_d}Q](\xi)
$$ 
for each $\xi$ in $\C^d$. If $P$ is not identically zero, then we always assume that $\sum_{|\alpha|=N} |a_{\alpha}|>0$, that is, the the total degree of the polynomial $P$ is $N$ -- in this case we say that the differential operator $P(\partial)$ is of order $N$.
\vskip.2cm

The following result shows that the functions $x\mapsto P(\partial)Q_x$ are linear combinations of moment functions on $X$.

\begin{thm}\label{poldifmom}
Let $X$ be a polynomial hypergroup in $d$ variables associated with the family of polynomials $\{Q_x:\,x\in X\}$. Then, for each multi-index $\alpha$ in $\N^d$, and for every $\lambda$ in $\C^d$, the function $x\mapsto [\partial^{\alpha}Q_x](\lambda)$ is a moment function sequence of rank $d$ associated with the exponential $x\mapsto Q_x$. 
\end{thm}

\begin{proof}
Let $c((x,y,t)$ denote the linearization coefficients of the polynomial hypergroup $X$, that is
$$
Q_x \cdot Q_y=\sum_{t\in X} c(x,y,t)Q_t
$$
for each $x,y$ in $X$. Further let
$f(x)=[\partial^{\alpha}Q_x](\lambda)$ whenever $x$ in $X$: then we have
$$
f(x*y)=\int_X f(t)\,d(\delta_x*\delta_y)(t)=\int_X [\partial^{\alpha}Q_t](\lambda) \,d(\delta_x*\delta_y)(t)=
$$
$$
\partial^{\alpha}[\int_X  Q_t(\lambda)\,d(\delta_x*\delta_y)(t)]=\partial^{\alpha}[\sum_{t\in X} c(x,y,t)Q_t(\lambda)]=
$$
$$
[\partial^{\alpha}(Q_x\cdot Q_y)](\lambda)=\sum_{\beta\leq \alpha} \binom{\alpha}{\beta} [\partial^{\beta}Q_x](\lambda)\cdot [\partial^{\alpha-\beta}Q_y](\lambda),
$$
which proves the statement.
\end{proof}

The moment functions of the form $x\mapsto \partial Q_x(\lambda)$ play an important role in our work: our purpose is to show that in any variety the linear combinations of these moment functions span a dense subspace. In other words, spectral synthesis holds on any polynomial hypergroup even when we restrict ourselves to exponential polynomials merely of the form $x\mapsto P(\partial)Q_x(\lambda)$. This will be proved in the subsequent paragraphs.

\section{The Fourier--Laplace transform}

In what follows we shall use the Fourier--Laplace transform on commutative hypergroups. Here we shortly summarize the basic concepts and results. Let $X$ be a commutative hypergroup and let $\mathcal C(X)$ denote the linear space of all complex valued continuous functions on $X$. Equipped with the uniform convergence on compact sets  $\mathcal C(X)$ is a locally convex topological vector space. If, for instance,  $X$ is discrete, then $\mathcal C(X)$ is the linear space of all complex valued functions on $X$ and the topology on $\mathcal C(X)$ is the topology of pointwise convergence. 
\vskip.2cm

The topological dual space of $\mathcal C(X)$ can be identified with the space of all compactly supported complex Borel measures on $X$, denoted by $\mathcal M_c(X)$. The identification depends on the Riesz Representation Theorem (see e.g. \cite[Theorem 6.19]{MR924157}): every continuous linear functional $\Lambda$ on $\mathcal C(X)$ can be represented in the form
$$
\Lambda(f)=\int_X f\,d\mu
$$ 
whenever $f$ is in $\mathcal C(X)$, where $\mu$ is a uniquely determined measure in $\mathcal M_c(X)$, depending on $\Lambda$ only. Clearly, if $X$ is discrete, then $\mathcal M_c(X)$ is the linear space of all finitely supported complex valued functions on $X$. For each measure $\mu$ in $\mathcal M_c(X)$ we define the measure $\widecheck{\mu}$ by
$$
\int_X f\,d\widecheck{\mu}=\int_X \widecheck{f}\,d\mu
$$
whenever $f$ is in $\mathcal C(X)$.
\vskip.2cm

The convolution in $X$ induces an algebra structure on $\mathcal M_c(X)$ in the following manner. Given two measures $\mu,\nu$ in $\mathcal M_c(X)$ their convolution is introduced as the measure $\mu*\nu$ defined on $\mathcal C(X)$ by the formula
$$
\int_X f\,d(\mu*\nu)=\int_X \int_X f(x*y)\,d\mu(x)\,d\nu(y),
$$
whenever $f$ is in $\mathcal C(X)$. We recall that $f(x*y)$ stands for the integral 
$$
\int_X f\,d(\delta_x*\delta_y)=\int_X f(t)\,d(\delta_x*\delta_y)(t),
$$
hence the above formula can be written in the more detailed form as
$$
\int_X f\,d(\mu*\nu)=\int_X \int_X \int_X f(t)\,d(\delta_x*\delta_y)(t)\,d\mu(x)\,d\nu(y),
$$
whenever $f$ is in $\mathcal C(X)$. As all the measures $\mu,\nu,\mu*\nu$ are compactly supported, this integral exists for each continuous function $f$.
\vskip.2cm

It is easy to check that the linear space $\mathcal M_c(X)$ is a commutative algebra with the convolution of measures. We call $\mathcal M_c(X)$ the measure algebra of the hypergroup $X$. If $X$ is discrete, then it is usually called the hypergroup algebra of $X$ -- imitating the terminology used in group theory.
\vskip.2cm

The Fourier--Laplace transform on $\mathcal M_c(X)$ is defined as follows: for each measure $\mu$ in $\mathcal M_c(X)$ and for each exponential $m$ on $X$ we let 
$$
\widehat{\mu}(m)=\int_X \widecheck{m}\,d\mu.
$$
Then $\widehat{\mu}:\mathcal E(X)\to\C$ is a complex valued function defined on the set $\mathcal E(X)$ of all exponentials \hbox{on $X$.} Clearly, the mapping $\mu\mapsto \widehat{\mu}$ is linear, and it is also an algebra homomorphism, as it is shown by the following simple calculation:
$$
(\mu*\nu)\,\widehat{}\,(m)=\int_X \widecheck{m}\,d(\mu*\nu)=\int_X\int_X \widecheck{m}(x*y)\,d\mu(x)\,d\nu(y)=
$$
$$
=\int_X\int_X \widecheck{m}(x)\widecheck{m}(y)\,d\mu(x)\,d\nu(y)=
\int_X \widecheck{m}(x)\,d\mu(x) \cdot \int_X \widecheck{m}(y)\,d\nu(y)=\widehat{\mu}(m)\cdot \widehat{\nu}(m),
$$
which holds for each $\mu,\nu$ in $\mathcal M_c(X)$ and for every exponential $m$ on $X$.
\vskip.2cm

From this property it follows that the set of all Fourier--Laplace transforms on $\mathcal E(X)$ form an algebra, which is called the Fourier algebra of $X$, denoted by $\mathcal A(X)$. In fact, by one of the most important properties of the Fourier--Laplace transform it follows, that the Fourier algebra is isomorphic to the measure algebra, which is expressed by the following theorem (see e.g. \cite[2.2.4 Theorem]{MR1312826}).

\begin{thm}\label{inj}
Let $X$ be a commutative hypergroup and let $\mu$ be in the measure algebra $\mathcal M_c(X)$. If $\widehat{\mu}(m)=0$ for each exponential $m$ on $X$, then $\mu=0$. 
\end{thm}

For our purposes it will be necessary to describe the Fourier algebra of polynomial hypergroups. This will be done in the subsequent paragraphs.
\vskip.2cm

Let $X$ be a polynomial hypergroup in $d$ variables generated by the family $\{Q_x:\,x\in X\}$ of polynomials. We may assume that $Q_o=1$, where $o$ is the identity of $X$. It is known that the function $m:X\to\C$ is an exponential on $X$ if and only if there exists a $\lambda$ in $\C^d$ such that 
$$
m(x)=Q_x(\lambda)
$$
holds for each $x$ in $X$. As $\lambda$ is obviously uniquely determined by $m$, hence the set $\mathcal E(X)$ of all exponentials on $X$ can be identified with $\C^d$. Consequently, the Fourier--Laplace transform of each measure in $\mathcal M_c(X)$ is a complex valued function on $\C^d$:
$$
\widehat{\mu}(\lambda)=\int_X Q_x(\lambda)\,d\mu(x).
$$
In fact, the integral is a finite sum, which implies that $\widehat{\mu}$ is a complex polynomial in $d$ variables. Conversely, let $P$ be any polynomial in the polynomial ring $\C[z_1,z_2,\dots,z_d]$. Then, by definition, $P$ has a representation of the form
$$
P(z)=\sum_{k=1}^n c_k Q_{x_k}(z)
$$
with some elements $x_k$ in $X$ and complex numbers $c_k$ for $k=1,2,\dots,n$. We have
$$
\widehat{\delta}_{x_k}(z)=\int_X Q_x(z)\,d\delta_{x_k}(x)=Q_{x_k}(z),
$$
hence 
$$
P=\sum_{k=1}^n c_k Q_{x_k}=\sum_{k=1}^n c_k \widehat{\delta}_{x_k}=\bigl(\sum_{k=1}^n c_k \delta_{x_k}\bigr)\,\widehat{}.
$$
Here $\mu=\sum_{k=1}^n c_k \delta_{x_k}$ is in $\mathcal M_c(X)$, hence we have proved that each polynomial in $\C[z_1,z_2,\dots,z_d]$ is in $\mathcal A(X)$. In other words, the Fourier algebra of each polynomial hypergroup in $d$ variables is the polynomial ring in $d$ variables.

\section{Spectral synthesis via moment functions}

We shall use the following basic result (see \cite{MR1925796}, \cite[Theorem 6.9]{MR2978690}).

\begin{thm} (Ehrenpreis--Palamodov Theorem)\label{EPT}
Let $I$ be a primary ideal in the polynomial ring $\C[z_1,z_2,\dots,z_d]$, and let $V$ denote the set of all common zeros of all polynomials in $I$. Then there exists a positive integer $t$ such that, for $i=1,2,\dots,t$  there exist differential operators with polynomial coefficients of the form
$$
A_i(z,\partial)=\sum_{j} c_{i,j} p_j(z_1,z_2,\dots,z_d) \partial_1^{j_1}\partial_2^{j_2}\cdots \partial_d^{j_d}
$$
for $i=1,2,\dots,t$ with the following property: a polynomial $p$ in $\C[z_1,z_2,\dots,z_d]$ lies in the ideal $I$ if and only if the result applying $A_i(z,\partial)$ to $f$ vanishes on $V$ for $i=1,2,\dots,t$.  
\end{thm}

The following result can be found in \cite[Theorem 6.10]{MR2978690}, but for the sake of completeness we present it together with its proof. 

\begin{thm}\label{momspec}
Let $X$ be a polynomial hypergroup in $d$ variables, and let $V$ be a proper variety on $X$. Then the functions of the form $x\mapsto P(\partial)Q_x(\lambda)$ in $V$, where $P$ is a polynomial in $d$ variables, and $\lambda$ is in $\C^d$ such that the exponential $x\mapsto Q_x(\lambda)$ is in $V$, span a dense subspace in $V$.
\end{thm}

\begin{proof}
We use the fact that, for each variety on $X$, if $V^{\perp}$ denotes the orthogonal complement  of $V$, that is, the set of all measures $\mu$ in $\mathcal M_c(X)$ such that $\int_X f\,d\mu=0$ for each $f$ in $V$, then $V^{\perp}$ is an ideal in $\mathcal M_c(X)$. Similarly, for each ideal $I$, the annihilator $I^{\perp}$ is the set of all functions in $\mathcal C(X)$ such that $\int f\,d\mu=0$ for each $\mu$ in $I$, and it is a variety in $\mathcal C(X)$. In addition, we have $V^{\perp\perp}=V$. (see e.g. \cite[Theorem 5, Theorem 16]{MR4217968}). In our case, in particular, $V^{\perp}$ is a proper ideal. By the Noether--Lasker Theorem,  (\cite[Theorem 7.13]{MR0242802}) $I$ is an intersection of (finitely many) primary ideals. It follows from the Ehrenpreis--Palamodov Theorem \ref{EPT} that, for each $\lambda$ in $V$there is a set $\mathcal P_{\lambda}$ of polynomials such that the measure $\mu$ in $\mathcal M_c(X)$ annihilates $V$ if and only if 
$$
[P(\partial)\widehat{\mu}](\lambda)=\int_X [P(\partial)Q_x](\lambda)\,d\mu(x)=0
$$
holds for each $\lambda$ in $V$ and $P$ in $\mathcal P_{\lambda}$. In other words, all exponential polynomials $x\mapsto [P(\partial)Q_x](\lambda)$ with $x\mapsto Q_x(\lambda)$ in $V$ and $P$ in $\mathcal P_{\lambda}$ belong to $V$, and their linear hull is dense in $V$. 
\end{proof}

From this theorem we infer that every exponential monomial is a linear combination of moment functions.

\begin{cor}\label{main}
Let $X$ be a polynomial hypergroup. Then every exponential polynomial on $X$ is a linear combination of moment functions contained in its variety.
\end{cor}

\begin{cor}
Let $X$ be a polynomial hypergroup in $d$ variables. Then every exponential polynomial on $X$ has the form $x\mapsto P(\partial)Q_x(\lambda)$ with some complex polynomial $P$ in $d$ variables and some complex number $\lambda$. 
\end{cor}

\begin{proof}
If $f$ is an exponential monomial, then it has a finite dimensional variety $V$. As $P(\partial)$ is a linear combination of differential operators of the form $\partial^{\alpha}$, and, by Theorem \ref{poldifmom}, $x\mapsto [\partial^{\alpha}Q_x](\lambda)$ is a moment function, the previous theorem implies that the linear combinations of all exponential monomials in $V$ span a dense subspace. By finite dimensionality, it means that $V$ is the linear span of all moment functions included in $V$, hence $f$ is a linear combination of moment functions.
\end{proof}


\end{document}